\documentclass[12pt]{amsart}
\usepackage{amsthm}
\usepackage{amsfonts,amssymb,amsmath}
\theoremstyle{definition}
\begin{document}
\title{Fake degrees of classical Weyl groups}
\author{William M. McGovern}
\begin{abstract}
We compute the fake degrees of representations of classical Weyl groups in terms of major indices of domino tableaux.
\end{abstract}
\maketitle

\section{Introduction}
Let $W$ be the complex reflection group $C_d\wr S_n$, where $C_d$ is the cyclic group of order $d$.  The action of $W$ on $\mathbb C^n$ by coordinate permutations and scalar multiplications by complex $d$th roots of unity then extends to the coordinate ring $S$ of $\mathbb C^n$, preserving the natural grading of $S$.  Let $I$ be the ideal of $S$ generated by $W$-invariant polynomials of positive degree.  The coinvariant algebra $C=S/I$ is then well known to be isomorphic to the regular representation of $W$; like $S$ it has a graded structure preserved by $W$.  Given an irreducible representation $\tau$ of $W$ of degree $d_\tau$ its so-called fake degree (polynomial) is the palindromic polynomial $f_\tau(q)=\sum_{i=1}^{d_\tau} q^{d_i}$, where the exponents $d_i$ are the degrees in which $\tau$ occurs in $C$, each listed according to its multiplicity.  There are well-known formulas for these degrees as powers of $q$ times ratios of products of differences $q^m-1$ for various $m$ (see \cite{Stei51,L77}).  More recently these formulas have been rewritten in terms of major indices of standard Young tableaux \cite{Sta71,Ste89}.  Here we give new formulas for these degrees for hyperoctahedral groups and Weyl groups of type $D$, using major indices of domino tableaux.  Such tableaux were first introduced in \cite{G90} to study primitive ideals in enveloping algebras of classical complex Lie algebras (see also \cite{G92,G93}).  They were used to study orbital subvarieties of nilpotent orbits in classical complex Lie algebras \cite{M21,M21'}.  We remark also that the notion of the major index of a domino tableau has been generalized to that of a descent of a border strip tableau in \cite{P21}.

\section{Types $B$ and $C$}

\noindent We begin with a quick review of the $q$-analogues of integers, factorials, and multinomial coefficients.  For $n$ a nonnegative integer, $k$ a positive integer at most equal to $n$, and $\alpha=(\alpha_1,\ldots,\alpha_m)$ a partition of $n$, set 
\vskip .1in
\[
[n]_q=1+q+\cdots+q^{n-1} = \frac {q^n-1} {q-1} \text{ for }n\ge1, [0]_q=1
\]
\[
[n]_q! = [n]_q[n-1]_q\cdots [1]_q,{n \choose k}_q=\frac {[n]_q!} {[k]_q![n-k]_q!}
\]
\[
{n\choose\alpha}_q=\frac {[n]_q!} {[\alpha_1]_q!\cdots[\alpha_m]_q!}
\]
\vskip .1in
\noindent Identifying $\alpha$ with the Young diagram of the corresponding shape, so that $\alpha_i$ is the length of the $i$th row of this diagram, denote by $h_c$ the length of the hook of the cell $c\in\alpha$.   Set $b(\alpha)=\sum_{i=1}^m (i-1)\alpha_i$.  

Recall that a {\sl standard Young tableau} $T$ of shape $\alpha$ is a bijective filling of the cells of $\alpha$ by the numbers from 1 to the sum $|\alpha|$ of the parts of $\alpha$ such that labels increase to the right in rows and down columns.  The {\sl major index} maj$(T)$ of $T$, sometimes just called the index of $T$, is the sum of the labels $i$ such that $i+1$ appears in a lower row than $i$ in $T$.  Denoting by SYT$(\alpha)$ the set of standard Young tableaux of shape $\alpha$, we have the generating function
\[
\text{SYT}(\alpha)^{\text {maj}}(q):=\sum_{T\in\text{SYT}(\alpha)}q^{\text{maj}(T)}
\]

It is well known that irreducible representations of $W$ are parametrized by ordered $d$-tuples $\lambda=(\lambda^{(1)}\ldots,\lambda^{(d)})$ of partitions $\lambda^{(i)}$ such that $\sum_i|\lambda^{(i)}|=n$ \cite[Thm. 4.1]{Ste89}.  Denote by $V_\lambda$ the representation corresponding to $\lambda$ and write $b(\lambda)=\sum_{i=1}^d (i-1)|\lambda^{(i)}|$.  A standard (Young) tableau $T$ of shape $\lambda$ is a $d$-tuple $(T^{(1)},\ldots,T^{(d)})$ of fillings of shapes $\lambda^{(1)},\ldots,\lambda^{(d)}$ such that the labels $1,\ldots,n$ are each used exactly once overall and labels increase across rows and down columns of each $T^{(i)}$.  The major index maj$(T)$ of $T$ is the sum of the labels $i$ such that either $i$ appears in a higher row than $i+1$ in the same filling $T^{(j)}$, or $i,i+1$ appear in the fillings $T^{(j)},T^{(k)}$, respectively, with $j<k$.  Then Stanley and Stembridge have derived the following formula for the fake degree $f_\lambda$ corresponding to $\lambda$ \cite{Sta71,Sta79},\cite[Thm. 5.3]{Ste89}.  Denote by SYT$(\lambda)$ the generating function $\sum_T q^{\text{maj}(T)}$, where the sum runs over standard tableaux of shape $\lambda$.

\newtheorem{theorem}{Theorem}
\newtheorem{corollary}{Corollary}
\begin{theorem}
The fake degree $f_\lambda$ corresponding to $\lambda$ is given by 
\[
f_\lambda=q^{b(\lambda)}\text{SYT}(\lambda)(q^d)=q^{b(\lambda)}{n\choose {|\lambda^{(1)}|,\ldots,|\lambda^{(d)}|}}_q\cdot\prod_{i=1}^d\text{SYT}(\lambda^{(i)})^{\text{maj}}(q^d)
\]
\noindent where
\[
\text{SYT}(\alpha)^{\text{maj}}(q)=\frac {q^{b(\alpha)}[r]_q!} {\prod_{c\in\alpha} [h_c]_q}
\]
\noindent for a partition $\alpha=(\alpha_1,\alpha_2,\ldots)$ of $r$ and $f_\lambda$ denotes the fake degree of the representation $V_\lambda$ corresponding to $\lambda$.  Equivalently, the multiplicity of $V_\lambda$ in the $k$-th graded piece of the coinvariant algebra $C$ is the number of standard tableaux $T$ of shape $\lambda$ with $k=b(\lambda)+d\,\text{maj}(T)$.
\end{theorem}

\noindent We now specialize down to the case $d=2$.  Given an ordered pair $(\lambda^{(1)},\lambda^{(2)})$ of partitions with $|\lambda^{(1)}|+|\lambda^{(2)}|=n$, we follow Lusztig \cite[\S3]{L77} to produce a single partition $\rho_1$ of $2n$, as follows (see also \cite{C85}).  Add zeroes to the parts of $\lambda^{(1)},\lambda^{(2)}$ as necessary to make $\lambda^{(1)}=(\alpha_1,\ldots,\alpha_{m+1})$ have exactly one more part than $\lambda^{(2)}=(\beta_1,\ldots,\beta_m)$.  For $1\le i\le m+1$, put $\alpha_i^*=\alpha_i+m+1-i$; similarly for $1\le j\le m$ put $\beta_j^*=\beta_j+m-j$.  Then the $\alpha_i^*$ and the $\beta_j^*$ are distinct.  Now set $\gamma_i=2\alpha_i^*,\delta_i=2\beta_i^*+1$, and combine and rearrange the $\gamma_i,\delta_i$ to make a partition $\rho_1'=(p_1',\ldots,p_r')$.  Then for $1\le i\le r$ set $p_i=p_i'-r+i$, thereby obtaining $\rho_1=(p_1,\ldots,p_r)$.  In a similar way we also use the $\alpha_i^*$ and $\beta_i^*$ to produce a single partition $\rho_2$ of $2n+1$, by putting $\gamma_i'=2\alpha_i^*+1,\delta_i'=2\beta_i^*$ and combining and rearranging the $\gamma_i',\delta_i'$ to make $\rho_2'=(q_1',\ldots,q_r')$, finally setting $q_i=q_i'-r+i$ to obtain $\rho_2=(q_1,\ldots,q_r)$.  The partitions $\rho,\rho_2$ that arise in this way are exactly those supporting a standard domino tableau of that shape.

Let $\alpha$ be a partition of $2n$. Recall from \cite{G90} that a {\sl domino tableau $T$ of shape $\alpha$} is an arrangement with shape $\alpha$ of $n$ nonoverlapping dominos, each horizontal or vertical.  Such a tableau becomes {\sl standard} if each domino is labelled by an integer between 1 and $n$ such that labels increase across rows and down columns and that every integer between 1 and $n$ occurs exactly once as a label.  If instead $\alpha$ is a partition of $2n+1$, then a domino tableau of shape $\alpha$ is an an arrangement with shape $\alpha$ of $n$ dominos together with a single square in the upper left corner.  It becomes standard if the dominos are labelled $1,\ldots,n$ obeying the same rules and the square is labelled 0.  The major index maj$(T)$ of a standard domino tableau $T$ is defined to be the sum of the labels $i$ such that both squares of the domino labelled $i$ in $T$ lie strictly above both squares of the domino labelled $i+1$.  Denote by SDT$(\alpha)$ the set of standard domino tableaux of shape $\alpha$ and by SDT$(\alpha)^{\text{maj}}(q)$ the generating function $\sum_{T\in\text{SDT}(\alpha)} q^{\text{maj}(T)}$.

\begin{theorem}
Take $d=2$ and let the partition pair $\lambda=(\lambda^{(1)},\lambda^{(2)})$ correspond as above to the partitions $\rho_1,\rho_2$ of $2n,2n+1$, respectively.  Then we have
\[
f_\lambda=q^{b(\lambda)}\text{SDT}(\rho_1)^{\text{maj}}(q^2)=q^{b(\lambda)}\text{SDT}(\rho_2)^{\text{maj}}(q^2)
\]
\end{theorem}

\begin{proof}
We construct bijections $\pi_C,\pi_B$ from the sets of standard domino tableaux of shapes $\rho_1,\rho_2$, respectively to the set of tableau pairs of shape $\lambda$ and then modify these to bijections $\pi_C',\pi_B'$ preserving major indices.

First we define $\pi_C$.  A standard domino tableau $T$ is built from the empty tableau in stages, at the $i$th of which a domino labelled $i$ is added to a standard tableau $T_{i-1}$ with $i-1$ dominos to make a new domino tableau $T_i$.  Assuming inductively that the pair $(Y_1,Y_2)$ of Young tableaux corresponding to $T_{i-1}$ has already been constructed, we will show how to add a single cell $c_i$ labelled $i$ to one of the $Y_i$ to make a new tableau pair. 

Suppose first that the domino $D_i$ labelled $i$ in $T_i$ is horizontal.
\vskip .2in
\begin{enumerate}
\item 
If $D_i$ lies in row $2m$ with its rightmost square in an even column then $c_i$ is added to the (end of the) $m$th row of $Y_2$.

\item
If $D_i$ lies in row $2m$ with its rightmost square in an odd column then $c_i$ is added to the $m$th row of $Y_1$.

\item
If $D_i$ lies in row $2m+1$ with its its rightmost square in an even column then $c_i$ is added to the $(m+1)$st row in $Y_1$.

\item
If $D_i$ lies in row $2m+1$ with its rightmost square in an odd column then $c_i$ is added to the $m$th row of $Y_2$ (or the first row, if $m=0$).
\end{enumerate}
\vskip .2in
\noindent Similarly, if instead $D_i$ is vertical, then
\vskip .2in
\begin{enumerate}
\item
If $D_i$ lies in an even column $2m$ with its lowest square in an even row, then $c_i$ is added to the $m$th column of $Y_1$.
\item
If $D_i$ lies in an even column $2m$ with its lowest square in an odd row, then $c_i$ is added to the $m$th column of $Y_2$.
\item
If $D_i$ lies in an odd column $2m+1$ with its lowest square in an even row, then $c_i$ is added to the $(m+1)$st column of $Y_2$.
\item
If $D_i$ lies in an odd column $2m+1$ with its lowest square in an odd row, then $c_i$ is added to the $(m+1)$st column of $Y_1$,  
\end{enumerate}
\vskip .2in
\noindent Next we define $\pi_B$, again proceeding inductively.  A domino tableau is constructed as before, but this time starting with a single square labelled 0.  Defining $T_{i-1},T_i$ as above and again letting $D_i$ be the domino labelled $i$ in $T_i$, assume first that $D_i$ is horizontal.
\vskip .2in
\begin{enumerate}
\item
If $D_i$ lies in an even row $2m$ with its rightmost square in an even column, then $c_i$ is added to the $(m+1)$st row of $Y_1$.
\item
If $D_i$ lies in an even row $2m$ with its rightmost square in an odd column, then $c_i$ is added to the $m$th row of $Y_2$.
\item
If $D_i$ lies in an odd row $2m+1$  with its rightmost square in an even column, then $c_i$ is added to the $m$th row of $Y_2$ (or to the first row, if $m=0$).
\item
If $D_i$ lies in an odd row $2m+1$ with its rightmost square in an odd column, then $c_i$ is added to the $(m+1)$st row of $Y_1$.
\end{enumerate}
\vskip .2in
If instead $D_i$ is vertical then
\vskip .2in
\begin{enumerate}
\item
If $D_i$ lies in an even column $2m$ with its lower square in an even row, then $c_i$ is added to the $m$th column of $Y_1$.
\item
If $D_i$ lies in an even column $2m$ with its lower square in an odd row, then $c_i$ is added to the $(m+1)$st column of $Y_2$.
\item
If $D_i$ lies in an odd column $2m+1$ with its lower square in an even row, then $c_i$ is added to the $m$th column of $Y_1$ (or the first column, if $m=0$).
\item
If $D_i$ lies in an odd column $2m+1$ with its lower square in an odd row, then $c_i$ is added to the $(m+1)$st column of $Y_2$.
\end{enumerate}
\vskip .2in
Let $\rho_1$ be a partition of $2n$ whose shape supports a domino tableau.  it is straightforward to check that if $T$ is a standard domino tableau of this shape, then the image $\pi_C(T)$ is a (Young) tableau pair $(Y_1,Y_2)$ such that the respective shapes $\lambda^{(1)},\lambda^{(2)}$ of $Y_1,Y_2$ form a pair corresponding to $\rho_1$ by the above recipe.  Similarly if $\rho_2$ is a partition of $2n+1$ whose shape supports a domino tableau and $T$ is a standard domino tableau of this shape, then $\pi_B(T)$ is a pair $(Y_1,Y_2)$ whose shapes $(\lambda^{(1)},\lambda^{(2)})$ correspond to $\rho_2$.

But now the major indices of $\pi_C(T),\pi_B(T)$ do not generally match that of $T$.  Instead, in type $C,m = \text{maj}(\pi_C(T)$ is given by the following rule: it is the sum of the indices $i$ such that
 $i$ lies in a strictly higher row within its tableau than $i+1$, or in the same row of their tableaux with the column of $i+1$ strictly to the left of that of $i$, or else $i,i+1$ lie in the same row and column of their tableaux with $i$ in $Y_1, i+1$ in $Y_2$.  Call this last condition $(\ast)$.  Running through the indices $i=1,\ldots,n-1$ in turn, we then produce a new tableau pair $(Y_1',Y_2')$ by flipping the labels $i$ and $i+1$ whenever either the indices $i,i+1$ satisfy $(\ast),i$ lies in $Y_2$, and $i$ in $Y_1$, or else $i,i+1$ do not satisfy $(1), i$ lies in $Y_1$, and $i+1$ lies in $Y_2$.  (One can check that, had the indices $i,i+1$ originally been in their current positions, then they would have been flipped, so that no two tableau pairs $(Y_1,Y_2)$ can yield the same pair $(Y_1',Y_2')$.)  Having run through the indices once, we then run through them again, flipping pairs of adjacent indices as before, except that we do not flip a pair of indices that was flipped previously.  We repeat this procedure until we get a pair $(Z_1,Z_2)$ of tableaux whose major index is exactly the sum of the indices contributing to the major index of $T$, so that maj$(Z_1,Z_2)=\text{maj}(T)$.  The map sending $(Y_1,Y_2)$ to $(Z_1,Z_2)$ is then a bijection.  The result follows in type $C$, setting $\pi_C'(T)=(Z_1,Z_2)$.
\vskip .2in
\noindent For example, if
$$
(Y_1,Y_2)=\bigg(\begin{pmatrix} 4\\6 \end{pmatrix}\,,\,\begin{pmatrix} 1&3\\2&5\end{pmatrix}\bigg)
$$
\noindent then we interchange first the 3 and the 4, then the 5 and the 6, obtaining 
$$
(Y_1',Y_2')=\bigg(\begin{pmatrix} 3\\ 5 \end{pmatrix}\,,\,\begin{pmatrix} 1&4\\2&6\end{pmatrix}\bigg)
$$
\noindent and then we interchange the 4 and 5, obtaining finally
$$
(Z_1,Z_2) = \bigg(\begin{pmatrix} 3\\4 \end{pmatrix}\,,\,\begin{pmatrix} 1&5\\2&6 \end{pmatrix}\bigg)
$$
\noindent If 
$$
(Y_1,Y_2)=\Bigg(\begin{pmatrix} 1\\3\\4\end{pmatrix}\,,\,\begin{pmatrix} 2\end{pmatrix}\Bigg)
$$
\noindent then we interchange first the 2 and the 3, then the 3 and the 4, to obtain
$$
(Z_1,Z_2)=\Bigg(\,\begin{pmatrix} 1\\2\\3 \end{pmatrix}\,\,,\,\,\begin{pmatrix} 4\end{pmatrix}\Bigg)
$$
\vskip .2in
Similarly, given a pair $(Y_1,Y_2)=\pi_B(T)$, we now find that $m=\text{maj}(T)$ is the sum of the indices $i$ such that $i,i+1$ lie in the same tableau with $i$ strictly higher in this tableau, or
 $i$ lies in $Y_1,  i+1$ in $Y_2$, with the row of $i$ higher than or equal to that of $i+1$,
 or else they lie in the same rows of their respective tableaux with the column of $i$ weakly to the left of that of $i+1$.  Call this last condition $(\ast\ast)$.  Running through the indices $1,\ldots,n-1$ in order, as in type $C$, we then flip the indices $i$ and $i+1$ whenever either $i,i+1$ satisfy $(\ast\ast),i$ is in $Y_2$, and $i+1$ is in $Y_1$, or else $i,i+1$ do not satisfy $(\ast\ast), i$ lies in $Y_1$, and $i+1$ lies in $Y_2$.  This time it is only necessary to run through the indices once, obtaining a tableau pair $(Z_1,Z_2)$ whose major index agrees with that of $T$.   The map sending $(Y_1,Y_2)$ to $(Z_1,Z_2)$ is again a bijection and the result follows in type $B$, setting $\pi_B'(T)=(Z_1,Z_2)$.
\end{proof}

Recall from \cite{L82,L86} that given any irreducible representation $V$ of $W$ there is a unique special representation $S$ occurring in the unique double cell of $W$ having $V$ as a subrepresentation.

\begin{corollary}
With notation as above, assume that $\mu=(\mu^{(1)},\mu^{(2)})$ is the partition pair corresponding to the special representation corresponding to $V_\lambda$.  Then the exponents $d_1,\ldots,d_r$ of $q$ in $f_\lambda$, counting multiplicities, are up to a uniform shift a subset of the corresponding exponents $e_1,\ldots,e_s$ for $V_\mu$.
\end{corollary}

\begin{proof}
The exponents $e_i$ are up to a uniform shift twice the major indices of the standard domino tableaux of shape $\rho_1$ or $\rho_2$, the partition of $2n$ or $2n+1$ corresponding as above to $\mu$.  A standard domino tableau $T$ of shape $\rho_1$ or $\rho_2$ can be moved through open cycles in the sense of \cite{G92} to have shape $\rho_1'$ or $\rho_2'$, the partition corresponding to $\lambda$. Moving through open cycles in this way preserves the $\tau$-invariant of $T$ in the sense of \cite{G92}, which determines its major index. More precisely, the index $i$ lies in the major index if and only if the difference $e_i-e_{i+1}$ of the $i$th and $(i+1)$st unit coordinate vectors in $\mathbb C^n$, regarded as a simple root in the standard root system of type $B_n$ or $C_n$, lies in the $\tau$-invariant of $T$.  Finally, the $\tau$-invariant of $T$ is an invariant of the Kazhdan-Lusztig left cell corresponding to $T$; this left cell $L$ is also the left cell corresponding to a suitable domino tableau of shape $\rho_1'$ or $\rho_2'$ \cite{G93}.  Hence the major indices of tableaux of shape $\rho_1$ or $\rho_2$, counting multiplicities, are also major indices of some tableau of shape $\rho_1'$ or $\rho_2'$.  The result follows.
\end{proof}

\noindent A weaker version of this result holds in type $D$; there the multiset of exponents is the union of two submultisets, each of them up to a uniform shift a subset of multiset of exponents for $\mu$ (but the shifts can be different for the two submultisets).

\section{Type $D$}

\noindent Let $W'$ be the subgroup of $W=C_2\wr S_n$ generated by coordinate permutations and evenly many sign changes.  Recall that irreducible representations of $W$ are parametrized by pairs $((\lambda^{(1)},\lambda^{(2)}),c)$, where $(\lambda^{(1)},\lambda^{(2)})$ is an {\sl unordered} pair of partitions with $|\lambda^{(1)}|+|\lambda^{(2)}|=n$ and $c=1$ if $\lambda^{(1)}\ne\lambda^{(2)}$ while $c=1$ or $2$ if $\lambda^{(1)}=\lambda^{(2)}$ \cite[Remark after Prop. 6.1]{Ste89}.  Given an unordered pair $\lambda=(\lambda^{(1)},\lambda^{(2)})$  with $\lambda^{(1)}\ne\lambda^{(2)}$, denote by $\lambda',\lambda''$ the respective ordered pairs $(\lambda^{(1)},\lambda^{(2)}),(\lambda^{(2)},\lambda^{(1)})$.  Write SYT'$(\lambda'),SYT''(\lambda'')$ for the respective generating functions $\sum_T q^{\text{maj}(T)}$ where the sum now ranges respectively over standard tableaux $T=(T^{(1)},T^{(2)})$ of shapes $\lambda',\lambda''$ such that in both cases such that the largest label occurs in $T^{(1)}$.  Then Stembridge has shown \cite[Cor. 6.4]{Ste89} (cf. also \cite[Thm. 2.35]{BKS20}) that

\begin{theorem}
With notation as above the fake degree $f_\lambda$ corresponding to $\lambda$ is given by
\[
f_\lambda(q)=q^{b(\lambda')}\text{SYT'}(\lambda')+q^{b(\lambda'')}\text{SYT}(\lambda'')
\]
\noindent If instead $\lambda=\lambda^{(1)}=\lambda^{(2)}$, then we have
\[
f_\lambda(q)=q^{b(\lambda)}\text{SYT}'(\lambda)
\]
\noindent for either of the representations corresponding to $(\lambda,\lambda)$, summing as above over standard tableaux $(T^{(1)},T^{(2)})$ with $n$ occurring in $T^{(1)}$ to define SYT'$(\lambda)$. 
\end{theorem}
\vskip .2in
Alternatively, a simple calculation leads to the following formula.  Instead of summing over standard tableaux $T$ of shape either $\lambda'$ or $\lambda"$, one can sum over standard tableaux of shape $\lambda'$ only, attaching the term $q^{b(\lambda')+\text{maj}(T)}$ to $T$ if the largest label $n$ occurs in $T^{(1)}$ and the term $q^{b(\lambda')+\text{maj}(T)-n}$ to $T$.  Thus the fake degrees attached to $\lambda$ in type $D$ are obtained from those in type $C$ attached to the ordered pair $\lambda'$ by subtracting $n$ from some of them.

Now let $\rho',\rho''$ be the partitions of $2n$ corresponding as above to $\lambda',\lambda''$.  As an immediate consequence of this theorem and the proof of the preceding one we get
 
 \begin{theorem}
 With notation as above we have
 \[
 f_\lambda(q)=q^{b(\lambda)'}\text{SDT}'(\lambda')(q^2)+q^{b(\lambda'')}\text{SDT}'(\lambda'')(q^2)
 \]
 \noindent where $\text{SDT}'(\lambda'),\text{SDT}'(\lambda'')$ denote the generating functions for standard domino tableaux $T$ of the respective shapes $\rho',\rho''$, weighted as above by their major indices, such that in both cases the pair $(Z_1,Z_2)=\pi_C'(T)$ has the largest label $n$ occurring in $Z_1$.
  If instead $\lambda=\lambda^{(1)}=\lambda^{(2)}$, then the right side is replaced by $q^{b(\lambda)}\text{SDT}'(\lambda)(q^2)$, again defining $\text{SDT}'(\lambda)$ by summing over domino tableaux $T$ such that $n$ occurs in the first coordinate $Z_1$ of the pair $\pi_C'(T)=(Z_1,Z_2)$.
 \end{theorem}
 \vskip .2in
 \noindent For example, take $\lambda=(\lambda^{(1)},\lambda^{(2)})=((1,1),1)$.  This pair corresponds to the partition $(2,2,2)$ of $6$; the complementary pair $((1),(1,1))$ corresponds to the partition $(2,2,1,1)$.  There are three standard domino tableaux of shape $(2,2,2)$, having major indices $1,2,3$.  The first two of these contribute to the sum in the theorem, leading to the terms $q^3,q^5$ in $f_\lambda$, given the shift by $q$ in this theorem.  There are three standard domino tableaux of shape $(2,2,1,1)$, of which only the one with major index 1 contributes to $f_\lambda$; since the shift is now by $q^2$, we get $f_\lambda=q^3+q^4+q^5$.  If $\lambda=(\lambda^{(1)},\lambda^{(2)})=((2),(2))$, then the corresponding partition is $(4,4)$; of the six standard domino tableaux of this shape, just three contribute to $f_\lambda$ and they have major indices $0,1,2$.  Here $f_\lambda = q^2+q^4+q^6$.
 
 In our first example above, where $\lambda^{(1)}=(1,1),\lambda^{(2)}=1$, applying the alternative formula using pairs of Young tableaux gives the degrees $d_i$ are 3,5, and $7-3=4$.  Alternatively, taking the ordered pair $((1),(1,1))$ we get that the $d_i$ are $6-3=3,4$, and $8-3=5$.  In the second example, taking $(\lambda^{(1)},\lambda^{(2)})=((2),(2))$, the $d_i$ are $2,4,6,6-4=2,8-4=4,10-4=6$. Cutting all multiplicities in half (in accordance with Theorem 4), we get that the $e_i$ are $2,4,6$.

\end{document}